\documentclass{amsart}
\usepackage{graphicx}
\usepackage{epstopdf}
\usepackage{pstricks}
\usepackage[mathscr]{eucal}
\usepackage{arydshln}

\theoremstyle{plain}
\newtheorem{prop}{Proposition}[section]

\newtheorem{conj}[prop]{Conjecture}
\newtheorem{lemm}[prop]{Lemma}

\newtheorem{thm}[prop]{Theorem}
\newtheorem{cor}[prop]{Corollary}
\newtheorem{ex}[prop]{Example}
\theoremstyle{definition}
\newtheorem{defn}[prop]{Definition}

\newtheorem{rem}[prop]{Remark}

\def\mcg#1;#2{\Gamma_{#1,#2}}
\def\fg#1;#2{\Pi_{#1,#2}}
\def\tb#1;#2{\mathscr{K}_{\frac{#1}{#2}}}

\begin{document}

\title[ Alexander   Polynomials of alternating braids ]
{Alexander   Polynomials of closed  alternating braids}

\author{Mark E. AlSukaiti}
\address{Department of Mathematical Sciences\\ College of Science\\ United Arab Emirates  University \\ 15551 Al Ain, U.A.E.}
\email{700035655@uaeu.ac.ae}

\author{Nafaa Chbili}
\address{Department of Mathematical Sciences\\ College of Science\\ United Arab Emirates  University \\ 15551 Al Ain, U.A.E.}
\email{nafaachbili@uaeu.ac.ae}
\urladdr{http://faculty.uaeu.ac.ae/nafaachbili}

\date{}

\begin{abstract}
We prove that the Alexander polynomials of certain families of  alternating 4-braid knots satisfy Fox's  Trapezoidal Conjecture. Moreover, we give explicit formulas  for the signature and for the first  4 coefficients of the Alexander polynomial for  a large
 family of alternating $n$-braid links and we verify that these 4 coefficients form a log-concave sequence.
\end{abstract}

\maketitle

\section{Introduction}


 An $n$-component link is an embedding of $n$ disjoint circles $\coprod S^{1}$ into the 3-dimensional sphere  $S^3$. A  knot is a one-component link.  A link diagram is a
regular planar  projection of the link together with hight information at each double point showing which of the two strands passes over the other. It is well known that the study of links up to isotopy is equivalent to the study of their regular projections up to Reidemeister moves. A link is said to be alternating if it admits a  diagram where
 the overpass and the underpass alternate as one travels along any strand of the diagram. Alternating links have been subject to extensive study, in particular their polynomial invariants are known to reflect the alternating nature of the diagram in a  clear way.\\
For any integer  $n\geq 2$, we let $B_n$ to denote  the  group of braids on $n$ strands. This group is generated by the elementary braids
$\sigma_{1}, \sigma_{2}, \dots ,\sigma_{n-1}$ subject to the
following relations:
\begin{align*}
\sigma_{i} \sigma_{j} & =\sigma_{j} \sigma_{i} \mbox{ if } |i-j| \geq 2\\
\sigma_{i} \sigma_{i+1} \sigma_{i} & = \sigma_{i+1}
\sigma_{i}\sigma_{i+1}, \ \forall \ 1 \leq i \leq n-2.
\end{align*}

By Alexander's theorem  \cite{Al1}, any   link $L$ in $S^3$
can be represented  as the closure of an $n$-braid $b$. We shall write   $L={\rm{cl}}(b)$. This representation is not unique and the minimum integer $n$ such that $L$ admits an $n$-braid representation
 is called the braid index of $L$. Recall that two braids represent the same link if and only if they are related by a finite sequence  of  Markov moves \cite{MuBook}.

The  Alexander polynomial \cite{Al2} is a topological   invariant of oriented links which associates with each link $L$ a Laurent polynomial with  integral coefficients
  $\Delta_L(t) \in \mathbb{Z}[t^{\pm 1/2}]$. Conway proved that this
polynomial can be defined in  a simple recursive way. The Alexander  polynomial can be also recovered from the Burau representation of the braid group \cite{Bur}.  More recently, Ozsv$\acute{a}$th and  Szab$\acute{o}$ \cite{OS}
 proved that this polynomial, up to the multiplication by some factor,  is   obtained as the Euler characteristic of the  link Floer homology. The Alexander polynomial of a knot $K$ is known to be  symmetric; it satisfies   $\Delta_K(t)=\Delta_K(t^{-1})$. Furthermore,  up to  the multiplication by a power of $t$  it is always  possible to write  $\Delta_K(t)=\displaystyle\sum_{i=0}^{2n}a_it^i$ with  $a_{0}\neq 0$ and  $a_{2n}\neq 0$.  \\
Murasugi \cite{Mu1958a,Mu1958b} studied the Alexander polynomial of alternating knots and proved that the coefficients of this polynomial satisfy  the condition $a_ia_{i+1}<0$ for all $0\leq i < 2n$.
Notice that this means that the Alexander polynomial of an alternating link can always be written in the from  $\Delta_K(t)=\pm \displaystyle\sum_{i=0}^{2n}a_i(-t)^i$, with $a_i>0$. Furthermore, the degree of $\Delta_K(t)$ is equal to twice  the genus of the knot \cite{Mu1958a,Mu1958b}. Fox asked whether one can characterize polynomials which arise as the Alexander polynomial of an alternating knot  \cite{Fo} and  conjectured that the coefficients of the Alexander polynomial of an alternating knot are  trapezoidal. In other words, these coefficients increase, stabilize then decrease in a symmetric way. In  \cite{St2005} Stoimenow conjectured that the coefficients of the Alexander polynomial of an alternating knot  are indeed log-concave. Equivalently,  they satisfy the condition $a_i^2\geq |a_{i+1}||a_{i-1}|$ for $0<i<2n$.  In \cite{Hu}, Huh listed Fox's Trapezoidal Conjecture (Stoimenow's version)   as one of the most interesting open problems about log-concavity. Through the past four  decades,  Fox's conjecture  has been confirmed  for several classes of alternating knots.  In particular, Hartley  proved the conjecture for 2-bridge knots  \cite{Ha}. Murasugi, showed that  the conjecture holds  for a large family of  alternating algebraic knots \cite{Mu1985}. Using knot Floer homology,  Ozsv$\acute{a}$th and  Szab$\acute{o}$ \cite{OS}  confirmed  the conjecture for alternating knots of genus 2. The same result has been obtained by Jong using  combinatorial methods \cite{Jo1,Jo2}. Hirasawa and Murasugi \cite{HM} showed  that the conjecture  holds  for stable alternating knots and suggested a refinement of this conjecture by stating that the length of the stable part  is less than the signature of the knot $\sigma(K)$. The refined form of this conjecture can be stated as follows:

\begin{conj}\label{conj}\cite{Fo,HM}
Let $K$ be an alternating knot and  $\Delta_K(t)=\pm \displaystyle\sum_{i=0}^{2n}a_i(-t)^i$, with $a_i>0$,   its  Alexander polynomial. Then there exists an integer $l \leq n$ such that:
$$a_0<a_1< \dots < a_{n-l/2}=\dots =a_{n+l/2}> \dots ...>a_{2n-1}>a_{2n}.$$ Moreover,  $l\leq |\sigma(K)|$.
\end{conj}

In  \cite{Ch}, Chen  proved   that this refinement holds for  two-bridge knots. Alrefai and the second author \cite{AC} showed that certain families of 3-braid knots satisfy Conjecture \ref{conj}. In \cite{AlC,SC}, the authors gave  explicit formulas for the   coefficients of the Alexander polynomial of weaving knots and proved that they satisfy the conditions of Conjecture \ref{conj}.
See also the  more recent results in \cite{HMV}. The main purpose of this paper is to confirm the conjecture for some families of alternating 4-braid knots.\\

Consider the alternating $n$-braid on $m$ blocks
$$\beta_{n,m}((p_{i,j}))=\Pi_{j=1}^m \left( \sigma_1^{p_{1,j}} \sigma_2^{-p_{2,j}} \sigma_3^{p_{3,j}} ... \sigma_{n-1}^{(-1)^{n-2} p_{{n-1},j}} \right),$$ where all $p_{i,j}>0$ or $p_{i,j}<0$  for all $i,j$. Obviously, the closure of $\beta_{n,m}((p_{i,j}))$ is an alternating link. It is worth mentioning here that this family of links contains the classes of weaving and  generalized hybrid weaving links whose quantum  invariants have been discussed in \cite{SC,SMR}. 

In this paper, we shall calculate the signature of  ${\rm{cl}}(\beta_{n,m}((p_{i,j})))$. Then, by studying  the Burau representation of the braid group, we shall check that for $m=1,2$ and $n=4$, the  Alexander polynomial of such closed braids satisfy Conjecture \ref{conj}. Moreover, we give explicit formulas for the first 4 coefficients of the Alexander polynomial of the link ${\rm{cl}}(\beta_{n,m}((p_{i,j})))$.

Here is an outline of this paper. In Section 2, we compute the signature of the closure of the braid $\beta_{n,m}(p_{i,j})$.  In Section 3,  we briefly recall the  Burau representation of the group $B_4$. In Section 4, we shall prove   that  certain classes of 4-braid links satisfy Conjecture \ref{conj}. Finally, in Section 5 we give explicit formulas for the first 4 coefficients of the Alexander
polynomial  of alternating $n$-braid links.
\parindent 0cm

\section{The signature of closed alternating braids}

This section  aims to calculate the signature of links that arise as the closure of certain  alternating braids on $n$ strands. Let us first prove the following lemma.

\begin{lemm}\label{matrixA}
	Let $A_n = \begin{bmatrix}
		2 	& -1 	& 0 	& ... & 0\\
		-1 	& 2 	& -1 	& \ddots	& \vdots\\
		0 	& \ddots& \ddots& \ddots& 0\\
		\vdots&	\ddots	& 		&		&  \\
		0	& ... 	&  0  	& -1 	& 2
	\end{bmatrix}$, be a matrix of order $n\times n$. Then $A_n$ has $n$ positive eigenvalues.
\end{lemm}

\begin{proof}
	We will prove the lemma  by showing that  $A_n$ is positive definite.
	
	Notice that $|A_1|= 2$ and $|A_2| = \begin{vmatrix}
		2 & -1 \\
		-1 &  2
	\end{vmatrix} = 3$.
	
	We now proceed by induction, assume that $|A_n| = n+1$ and $|A_{n-1}| = n$, then:
	\begin{align*}
		|A_{n+1}| &= \left|\begin{tabular}{c c c : c}
			& 		& 		& 0\\
			& $A_n$ 	& 		& \vdots \\
			&		& 		& -1 	\\
			\hdashline
			0 	& ... 	& 	-1	& 2
		\end{tabular}\right| \\
		&= 2 |A_n| + \left|\begin{tabular}{c c c : c}
			& 		& 		& 0\\
			& $A_{n-1}$ 	& 		& \vdots \\
			&		& 		& 0 	\\
			\hdashline
			0 	& ... 	& 	-1	& -1
		\end{tabular}\right| \\
		&= 2 |A_n| - |A_{n-1}| + \left|\begin{tabular}{c c c : c}
			& 		& 		& 0\\
			& $A_{n-2}$ 	& 		& \vdots \\
			&		& 		& 0 	\\
			\hdashline
			0 	& ... 	& 	-1	& 0
		\end{tabular}\right|\\
		&= 2(n+1) - n\\
		&= n+2.
	\end{align*}
	
	Therefore  $|A_n| = n+1$ and since the principle sub-matrices of $A_n$ are just $A_r$ for $r\leq n$ we get that  all principle sub-matrices of $A_n$ have  positive determinants.
 Thus,   $A_n$ is positive definite. So it  has $n$ positive eigenvalues.
\end{proof}

\begin{thm}
	Let $K$ be the closure of the $n$-braid $\beta_{n,m}((p_{i,j}))$, where either all $p_{i,j}>0$ or $p_{i,j}<0$  for all $i,j$.  Then the signature of $K$ is given by the following formulas.
\begin{enumerate}
\item If $n$ is odd, then
	$$\sigma(K) = \sum_{j=1}^m\sum_{i=1}^{\lfloor \frac{n}{2} \rfloor} p_{2i,j} - \sum_{j=1}^m\sum_{i=1}^{\lfloor \frac{n}{2}\rfloor} p_{2i-1,j}. $$
\item If $n$ is even and $p_{i,j}>0$, then
	$$\sigma(K) = \sum_{j=1}^m\sum_{i=1}^{\frac{n}{2}-1} p_{2i,j} - \sum_{j=1}^m\sum_{i=1}^{\frac{n}{2}} p_{2i-1,j} \ +1.$$
\item If n is even and $p_{i,j}<0$, then
	$$\sigma(K) = \sum_{j=1}^m\sum_{i=1}^{\frac{n}{2}-1} p_{2i,j} - \sum_{j=1}^m\sum_{i=1}^{\frac{n}{2}} p_{2i-1,j} \ -1.$$
\end{enumerate}
\end{thm}

\begin{ex}
	If $K$ is the closure of the braid $\sigma_1 ^ 3 \sigma_2^{-2} \sigma_3^5 \in B_4$, then the signature of $K$ is $\sigma(K) = 2 - (3+5) +1 = -5$. While, if
	 $K$ is the closure of the braid $\sigma_1 ^ 3 \sigma_2^{-4} \sigma_3^2 \sigma_4^{-9} \sigma_1 ^ 2 \sigma_2^{-1} \sigma_3^3 \sigma_4^{-8}  \in B_5$, then the signature of $K$ is $\sigma(K) = 4+9+1+8 - (3+2+2+3)= 12$.
\end{ex}

\begin{proof}
	We would prove this by using induction on the number of blocks $m$. Let $m=1$ and    $K$ be the knot we obtain from the closure of the braid $\sigma_1^{p_{1}} \sigma_2^{-p_{2}} \sigma_3^{p_{3}} ... \sigma_{n-1}^{(-1)^{n-2} p_{{n-1}}}$ such that $p_i > 0$ for all $i$, see Figure \ref{Braidknot1}.
	Let $S$ be the Seifert Surface of $K$ obtained through Seifert's algorithm and let $G$ be  the associated Seifert Graph \cite{MuBook}. Then $S$ and $G$ are  as shown in Figure \ref{SeifertSurface}. Notice that $G$ has $\sum_{i=1}^{n} (p_i -1)$ faces excluding the face that contains the point infinity. These faces  correspond to the loops $\gamma_{i,j}$ for $1 \leq i \leq n$ and $1 \leq j \leq p_i -1$ as shown in the Figure \ref{SeifertSurface}.
We have the following notable linking numbers:
	\begin{gather*}
		lk (\gamma_{i,j}, \gamma_{i,j}^*) = -1 \text{ for odd $i$,}\\
		lk (\gamma_{i,j}, \gamma_{i,j+1}^*) = 1 \text{ for odd $i$,}\\
		lk (\gamma_{i,j}, \gamma_{i,j}^*) = 1 \text{ for even $i$,}\\
		lk (\gamma_{i,j}, \gamma_{i,j+1}^*) = -1 \text{ for even $i$.}
	\end{gather*}
	
	Any other linking number is $0$.
	
	\begin{figure}[h]
		\centering
		\begin{minipage}{0.3\textwidth}
			\begin{center}
				\includegraphics[width =.9\linewidth]{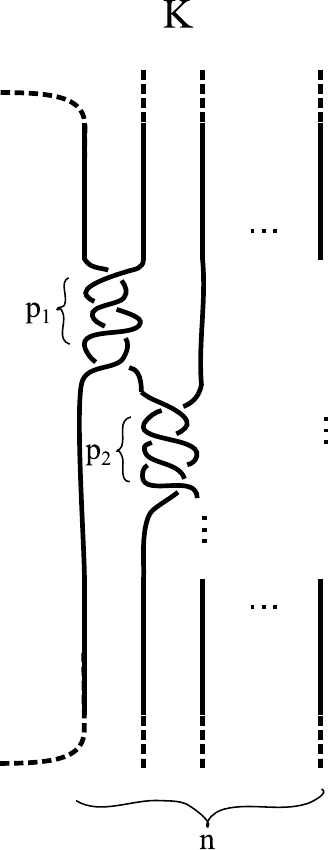}
			\end{center}
			\caption{Knot $K$ as a closure of the braid  $\Pi_{i=1}^{n-1} \sigma_i^{(-1)^{i-1} p_{i}}$. }
			\label{Braidknot1}
		\end{minipage}
		\hspace{1cm}
		\begin{minipage}{0.6\textwidth}
			\centering
			\includegraphics[width=\linewidth]{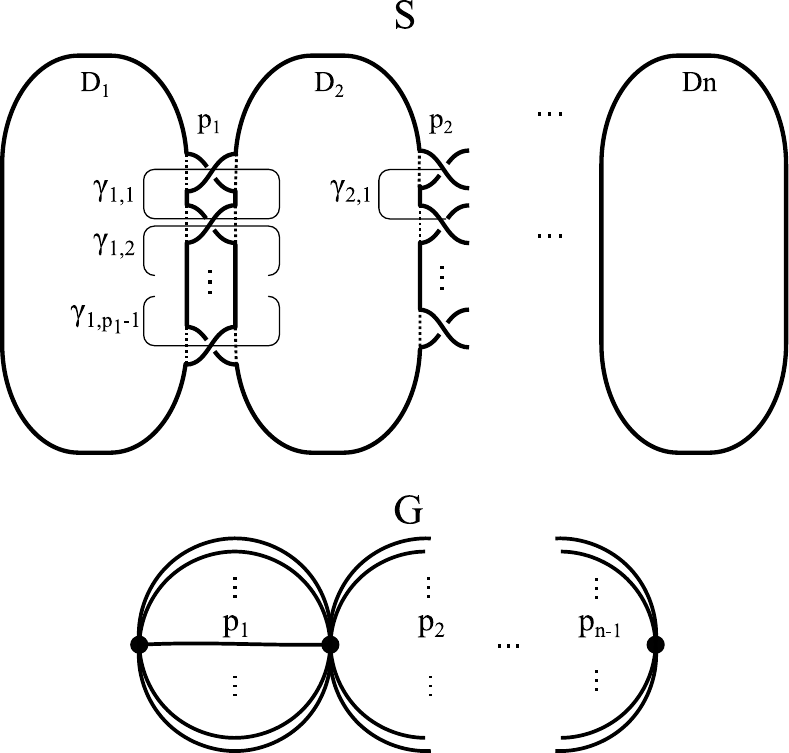}
			\caption{Seifert Surface $S$ of $K$, and its associated Seifert Graph $G$.}
			\label{SeifertSurface}
		\end{minipage}	
	\end{figure}

	If we define $A_n$ to be the $n\times n$ matrix as described in Lemma \ref{matrixA} and $M$ to be the Seifert matrix of $K$ then we have the following:
	
	$$ M + M^T= \begin{bmatrix}
		- A_{p_1 -1} & \\
		& A_{p_2 -1} & \\
		& & \ddots \\
		& & & (-1)^{n-1}A_{p_{n-1} -1}
	\end{bmatrix}.$$
	
	That is $M + M^T$ is a block matrix with the matrix $(-1)^{k}A_{p_k -1}$ on its main diagonal. Thus, the eigenvalues of $M + M^T$ are the eigenvalues of the respective $(-1)^{k}A_{p_k -1}$. Therefore,  $M + M^T$ has $\sum_{i = 1}^{\lfloor \frac{n}{2} \rfloor} (p_{2i}-1) $ positive eigenvalues and $\sum_{i = 1}^{\lfloor \frac{n}{2} \rfloor} (p_{2i-1}-1) $ negative eigenvalues if $n$ is odd, and $\sum_{i = 1}^{\frac{n}{2}-1} (p_{2i}-1) $ positive eigenvalues and $\sum_{i = 1}^{\frac{n}{2}} (p_{2i-1}-1) $ negative eigenvalues if $n$ is even. Thus proving our base case in the case $p_i>0$. Notice that the case $p_i<0$ is settled simply by taking the mirror image.\\
	\begin{rem}  It is worth mentioning here that the formula above for the signature  can be obtained by considering the link $K$
as a connected sum the  torus links $T(2,(-1)^{i-1} p_{i})$. This link has  signature $(-1)^{i} (p_{i}-1)$ if $p_{i}>0$ and $(-1)^{i} (p_{i}+1)$ if $p_{i}<0$. Recall that the
 signature of a connected sum of links is simply the sum of the signatures of these links.
	\end{rem}
Assume that   $n$ is  even and that  $p_{i,j}>0$ for all $i,j$.  Recall that  $\beta_{n,m-1}((p_{i,j}))$ stands for  the $n$-braid on $m-1$ blocks
 $$\Pi_{j=1}^{m-1} \left( \sigma_1^{p_{1,j}} \sigma_2^{-p_{2,j}} \sigma_3^{p_{3,j}} ... \sigma_{n-1}^{(-1)^{n-2} p_{{n-1},j}} \right).$$
	As our induction assumption, we assume that $$\sigma({\rm cl}(\beta_{n,m-1}((p_{i,j})))) = \sum_{j=1}^{m-1}\sum_{i=1}^{\frac{n}{2}-1} p_{2i,j} - \sum_{j=1}^{m-1}\sum_{i=1}^{\frac{n}{2}} p_{2i-1,j} \ +1.$$
	
Let $L$ be the alternating link obtained from the closure of the braid
$$
	\beta_{n,m}((p_{i,j}))=\beta_{n,m-1}((p_{i,j})) \left( \sigma_1^{p_{1,m}} \sigma_2^{-p_{2,m}} \sigma_3^{p_{3,m}} ... \sigma_{n-1}^{(-1)^{n-2} p_{{n-1},m}} \right)$$
with $m\geq 2$.  Let $c$ be a positive crossing of $L$ corresponding to an  elementary braid $\sigma_{n-1}$ in the last block. If  we smooth   $L=L_+$ at the crossing $c$,  then
 the resulting links $L_{v}$ and $L_{h}$   still remain alternating  and non-split, see Figure \ref{altpres}.
Consequently,  the links $L=L_{+}$, $L_{v}$ and  $L_{h}$  satisfy the conditions  $\det(L_v)>0, \det(L_h) >0$, and $\det(L_+) = \det(L_v) + \det(L_h)$.
Therefore,  by a result of   Manolescu and  Ozsv$\acute{a}$th \cite{MO}, we have:
	$$ \sigma(L_+) = \sigma(L_v) -1.$$

	\begin{figure}[h]
		\begin{center}
			\includegraphics[width = 0.8\linewidth]{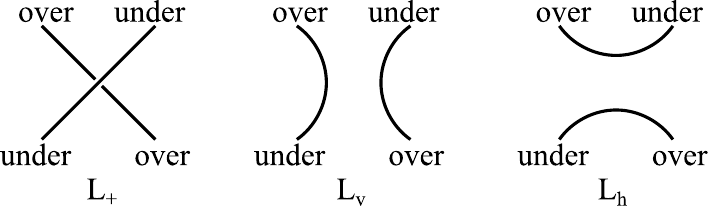}
		\end{center}
		\caption{Both $L_v$ and $L_h$ preserve the alternating property of the  diagram.}\label{altpres}
	\end{figure}
	Notice that  the relation above  holds for all crossings associated to $\sigma_{n-1}^{p_{n-1,m}}$. This implies that:
	$$ \sigma(L_+) = \sigma\left({\rm cl}\left(  \beta_{n,m-1}((p_{i,j})) \left( \sigma_1^{p_{1,m}} \sigma_2^{-p_{2,m}} \sigma_3^{p_{3,m}} ... \sigma_{n-2}^{(-1)^{n-3} p_{{n-2},m}} \right)   \right)\right) -p_{n-1,m}. $$
	
	Similar logic holds for all odd values of $i<n-1$. That is:
	$$ \sigma(L_+) = \sigma\left({\rm cl}\left(\beta_{n,m-1}((p_{i,j})) \left( \sigma_2^{-p_{2,m}} \sigma_4^{-p_{4,m}} ... \sigma_{n-2}^{(-1)^{n-3} p_{{n-2},m}} \right) \right)\right) -\sum_{i=1}^{\frac{n}{2}} p_{2i-1,m}. $$
	
	To find the signature of the link $K = {\rm cl}\left(\beta_{n,m-1}((p_{i,j})) \left( \sigma_2^{-p_{2,m}} \sigma_4^{-p_{4,m}} ... \sigma_{n-2}^{(-1)^{n-3} p_{{n-2},m}} \right) \right)$ we shall look at its mirror. Let $K^*$ be the mirror image of $K$ and $\beta_{n,m-1}((p_{i,j}))^*$ be the mirror of $\beta_{n,m-1}((p_{i,j}))$. Then we can apply the same logic above to get the following:
	\begin{align*}
		\sigma(K^*) &= \sigma \left( {\rm cl} \left(\beta_{n,m-1}((p_{i,j}))^* \right)\right) - \sum_{i=1}^{\frac{n}{2}-1} p_{2i,m} \\
		&= \sum_{j=1}^{m-1}\sum_{i=1}^{\frac{n}{2}-1} -p_{2i,j} - \sum_{j=1}^{m-1}\sum_{i=1}^{\frac{n}{2}} -p_{2i-1,j} \ -1 - \sum_{i=1}^{\frac{n}{2}-1} p_{2i,m}\\
		&= \sum_{j=1}^{m}\sum_{i=1}^{\frac{n}{2}-1} -p_{2i,j} - \sum_{j=1}^{m-1}\sum_{i=1}^{\frac{n}{2}} -p_{2i-1,j} \ -1.
	\end{align*}
	Finally getting:
	\begin{align*}
		\sigma(L_+) &= \sum_{j=1}^{m}\sum_{i=1}^{\frac{n}{2}-1} p_{2i,j} - \sum_{j=1}^{m-1}\sum_{i=1}^{\frac{n}{2}} p_{2i-1,j} \ +1 -\sum_{i=1}^{\frac{n}{2}} p_{2i-1,m}\\
		&= \sum_{j=1}^{m}\sum_{i=1}^{\frac{n}{2}-1} p_{2i,j} - \sum_{j=1}^{m}\sum_{i=1}^{\frac{n}{2}} p_{2i-1,j} \ +1.
	\end{align*}
	
	Thus completing our induction and proving our formula in  case $n$ is even. Notice that  to prove the formula  in  the case $p_{i,j}<0$ we just have to look at the mirror image of the link. A similar induction argument  could be used  to show that the theorem holds for odd values of $n$ as well.
\end{proof}

\begin{cor}
	Let $L$ be the closure of the 4-braid $\Pi_{i=1}^{m} \sigma_1^{p_i} \sigma_2^{-q_i} \sigma_3^{r_i}$, with $p_i, q_i, r_i,m >0$, then $\sigma(L)=Q - P - R +1$ where $Q = \sum_{i=1}^{m} q_i, P = \sum_{i=1}^{m} p_i,$ and $R = \sum_{i=1}^{m} r_i$.
\end{cor}
\section{The Alexander Polynomial and the Burau representation}
The Alexander polynomial \cite{Al2} is  a classical knot invariant which is of central importance to  knot theory. This topological  invariant  of oriented links  is a one-variable Laurent polynomial with integral coefficients
 $\Delta_L(t)$ which can  be defined in several different, but equivalent, ways.  In particular, it  is uniquely determined   recursively   using Conway skein relations:
$$\begin{array}{l}
\Delta_U(t)=1,\\
\Delta_{L_+}(t) - \Delta_{L_-}(t)= (\sqrt{t}-\displaystyle\frac{1}{\sqrt{t}}) \Delta_{L_0}(t),
\end{array}$$
where $U$ is the unknot and  $L_{+}$, $L_{-}$ and  $L_{0}$ are three oriented link diagrams
which are identical except in a small region  where they are as pictured in Figure \ref{figure1}.
\begin{figure}[h]
\centering
\includegraphics[width=10cm,height=2.5cm]{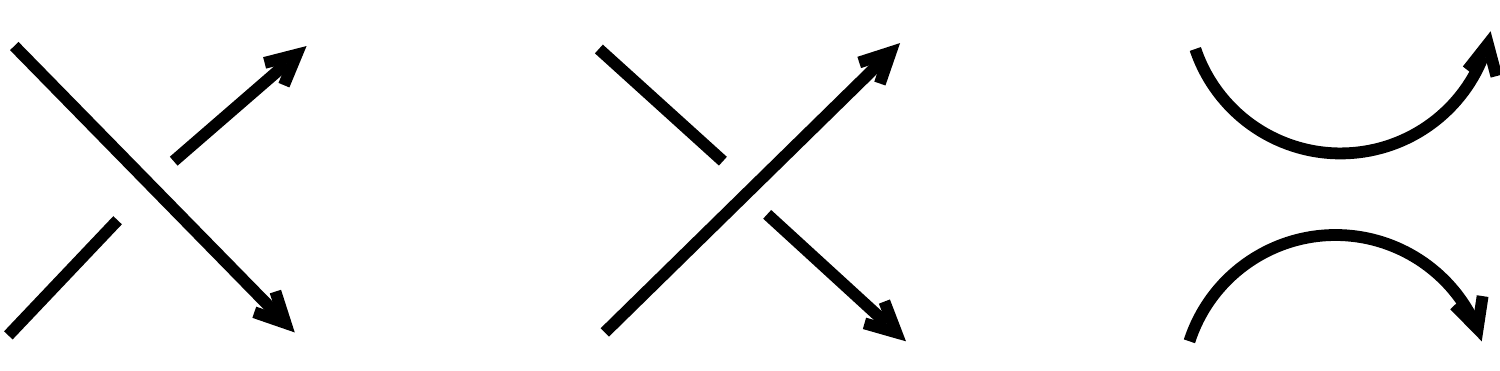}
\caption{aThe three oriented links  $L_+$, $L_{-}$ and  $L_0$,  respectively.}
\label{figure1}
\end{figure}
 This polynomial is known to be symmetric, in the sense that it satisfies  $\Delta_L(t)= \Delta_L(t^{-1})$. Another way to define the Alexander polynomial is through
 the reduced Burau representation of the braid group \cite{Bur}. Since we are only interested in 4-braid links, we shall restrict ourselves to the Burau representation of  the braid group   $B_4$. This representation is defined as follows. Let   $b$ be a given 4-braid and  $e_b$  the exponent sum of $b$ as a word in the
elementary braids $\sigma_1$, $\sigma_2$ and $\sigma_3$.   Let $\psi_{4,t}: B_{4}
\longrightarrow GL(3, \mathbb{Z}[t,t^{-1}])$ be the reduced  Burau
representation defined on the generators of $B_4$ by:

$\psi_{4,t}(\sigma_1 )$ = $\begin{pmatrix}
		-t & 1  & 0 \\
		0 & 1 & 0 \\
		0 & 0 & 1
	\end{pmatrix}$, $\psi_{4,t}(\sigma_2 )$ = $\begin{pmatrix}
		1 & 0 & 0 \\
		t & -t & 1 \\
		0 & 0 & 1
	\end{pmatrix}$ and $\psi_{4,t}(\sigma_3)$ = $\begin{pmatrix}
		1 & 0 & 0 \\
		0 & 1 & 0 \\
		0 & t& -t
	\end{pmatrix}.$

The Alexander polynomial, up to the multiplication by a unit of  $\mathbb{Z}[t^{1/2},t^{-1/2}]$,  of the link $L={\rm cl}(b)$ is obtained from the Burau representation  by the following formula:
$$\Delta_L(t)=(\frac{-1}{\sqrt{t}})^{e_b-3}\frac{1-t}{1-t^4}\mbox{det}(I_3-\psi_{4,t}(b)),$$
where $I_3$ is the $3\times 3$ identity matrix.

Recall that the coefficients of the  Alexander polynomial of an alternating knot  alternate in sign \cite{Mu1958b}. Throughout the rest of this paper,   we find it more convenient
  to make the substitution   $s=-t$. Hence, if $K$ is alternating then we can write   $\Delta_K(s)=\pm \displaystyle\sum_{i=0}^{2n}a_is^i$, with $a_i>0$.\\

\section{Main Result}

In this section we shall confirm  Conjecture \ref{conj} for two families of 4-braid knots. We start by proving some general properties of trapezoidal polynomials.
\begin{defn}
	Let $p(s) =\sum_{k=0}^{n} a_k s^{k+r}$ for $r \in \mathbb{Z}$ be a polynomial. We say that $p$ is symmetric if  we have
	$ a_k =  a_{n-k}$   for all $0 \leq k \leq n$.
\end{defn}

\begin{defn}
	Let $p(s) =\sum_{k=0}^{n} a_k s^{k+r}$ for $r \in \mathbb{Z}$ be a polynomial. We say that $p$ is trapezoidal if for some $0 \leq i \leq j \leq n$ we have that:
	$$ a_0 < a_1 < a_2 < ... a_i = a_{i+1} = ... = a_{j} > a_{j+1} > ... > a_n.$$
\end{defn}

\begin{defn}
	Let $p(s) = \sum_{k=0}^{n} a_k s^{k+r} $ for $r \in \mathbb{Z}$ be a symmetric trapezoidal polynomial such that $a_k \neq 0$ for all $0\leq k\leq n$. Then we define the center of $p(s)$ to be $\frac{2r+n}{2}$ and the length of the stable part of $p$ to be  the largest integer $l$ such that $a_{k_0} = a_{k_0 +1} = ... = a_{k_0+l}$ for some $k_0 \leq n$.
\end{defn}

For any positive integer $n$, let  $[n]=\frac{1-s^n}{1-s}$. Note here that   we consider $[n]$  to be a trapezoidal polynomial with center at $\frac{n-1}{2}$ and length of stable part to be $n-1$.

\begin{lemm} \label{trapezoidalprod}
If $p(s) = \sum_{k=0}^{m} a_k s^{k} $ is a symmetric trapezoidal polynomial with center $\frac{m}{2}$ and length of its stable part $l$, then the product $p(s) [n]$, with $m \geq n$, is a symmetric trapezoidal polynomial with its center at $\frac{m+n-1}{2}$.
	If $l \geq n$ then the length of its stable part is $l-n+1$. Otherwise, if $l < n$ and $m+n-1$ is even (i.e the degree of the polynomial) the length is $0$, else the length is $1$.
\end{lemm}

\begin{proof}
	The product of two symmetric polynomials is symmetric so $p(s) [n]$ is symmetric and it is clear why the center is at $\frac{m+n-1}{2}$. The rest of this proof will be spent proving that it is trapezoidal  and computing  its length.
	
	Let the coefficient of $s^k$ in $p(s) [n]$ be denoted by $c_k$, and $i_0$ be the integer such that $a_{i_0} = a_{i_0+1} = ... = a_{i_0 + l}$, then we have the following:
	$$ c_k = \sum_{i=k-n+1}^{k} a_i. $$
	
Assume that $l \geq  n$. Since $a_{i_0 + n-1} = a_{i_0+n} = ... = a_{i_0 + l}$ we get that $c_{i_0 + n-1} = c_{i_0 + n} = ... = c_{i_0 + l}$. So the length of the stable part is $l-n+1$. Moreover,  since $p(s)$ is trapezoidal, then if $0 \leq k < i_0 + n-1$ we have that $c_k < c_{k+1}$. Similarly if $ i_0 + l -1 < k < m+n-1$, then we have that $c_k > c_{k+1}$.
	
In the case  $l < n$, we notice that if $k < \frac{n+m-1}{2}$ then $c_k < c_{k+1}$ because  any $n$ consecutive coefficients of $p(s)$ would not stabilize. Since $p(s)[n]$ is symmetric, we get $c_k = c_{n+m-1-k}$. Observe that if $n+m-1$ is even then the length of our stable part is 0 (since our coefficients increase up to $c_\frac{n+m-1}{2}$ and are symmetric around it) and if it is odd then the length of our stable part is 1 since if we take $k = \frac{n+m}{2}$ we get that $c_\frac{n+m}{2} = c_{\frac{n+m}{2}-1}$. In conclusion the product  $p(s) [n]$ is always trapezoidal and the length of the stable part is as stated in the lemma.
\end{proof}

\begin{cor}
	If $p(s) = \sum_{k=0}^{m} a_k s^{k+r} $, for $r \in \mathbb{Z}$ is a symmetric trapezoidal polynomial with center $\frac{m}{2}$ and length of its stable part $l$, then $p(s) [n]$ is a symmetric trapezoidal polynomial with its center at $\frac{2r+m+n-1}{2}$.
	If $l \geq n$, then the length of its stable part is $l-n+1$. Otherwise, if $l < n$ and $m+n-1$ is even the length is $0$, else the length is $1$.
\end{cor}
Finally, it is easy to see that we have the following lemma.
\begin{lemm}\label{addsymm}
	Let $p(s) = \sum_{k=0}^{n} a_k s^{k+r_1}$ and $q(s) = \sum_{k=0}^{m} b_k s^{k+r_2}$ be two symmetric trapezoidal polynomials with length of  stable parts  $l_p$ and $ l_q$ respectively. If  both polynomials have the same center $c$, then $p(s) + q(s)$ is a symmetric trapezoidal polynomial with center $c$ and length of  stable part  $min\{l_p, l_q\}$.
\end{lemm}

\begin{thm}\label{case1}
	Let $L$ be the alternating link obtained through the closure of the 4-braid $\sigma_1^p \sigma_2^{-q} \sigma_3^r $, with  $p,q,r >0$. Then its Alexander polynomial  (up to some unit in $\mathbb{Z}[s^{1/2},s^{-1/2}])$ is given by:
	$$\Delta_L(s) = [p][q][r].$$
	Moreover $\Delta_L(s)$ is trapezoidal, and the length of its stable part is $l \leq |\sigma(L)|$.
\end{thm}

\begin{figure}[h]
	\begin{center}
		\includegraphics[width =0.7\linewidth]{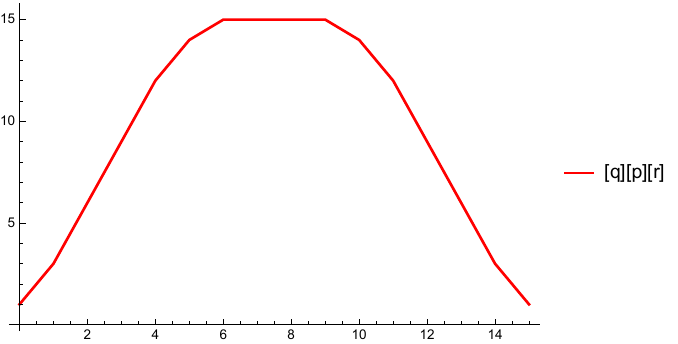}
	\end{center}
	\caption{Coefficients of $\Delta_K(s)$ plotted for a knot $K = cl(\sigma_1^p \sigma_2^{-q} \sigma_3^r)$.}
\label{Trapezoidal1}
\end{figure}
\begin{proof}
	Since the closure of the 4-braid $\sigma_1^p \sigma_2^{-q} \sigma_3^r  $ is just the connected sum of the torus links of type $(2,p)$, $(2,-q)$ and $(2,r)$ then its Alexander polynomial, up to some unit in $\mathbb{Z}[s^{1/2},-s^{1/2}]$, is given by:
	\begin{align*}
		\Delta_L(s) &= \Delta_{(2,p)}(s) \Delta_{(2,-q)}(s) \Delta_{(2,r)}(s)\\
		&= [p] [q][r].
	\end{align*}
	
	By Lemma  \ref{trapezoidalprod}, one can easily see that  $\Delta_L(s)$ is trapezoidal. Without loss of generality, we  can assume that $p \geq q \geq r$.  The length of  the
stable part of $\Delta_L(s)$ is given as follows.\\
	If $p-q-r+1 \geq 0$ then  $l = p-q-r+1$. Recall  that $|\sigma(L)| = |q-p-r+1| = p+r-q-1$. Obviously,  $p-q-r+1 \leq p+r-q-1$ since $r\geq 1$. Thus we get $l \leq |\sigma(L)|$.
	Otherwise if $p-q-r+1 < 0$ then notice that $\sigma(L)$ has the same parity as the degree of $\Delta_L(s)$ which is $p+q+r-3$. Hence,  if $\sigma(L)$ is even, then  we have $l = 0 \leq |\sigma(L)|$. Otherwise, if $\sigma(L)$ is odd then $l = 1 \leq |\sigma(L)|$.
\end{proof}
\begin{thm}
	Let $L$ be the alternating link obtained through closure of the 4-braid
	 $\sigma_1^{p_1} \sigma_2^{-q_1} \sigma_3^{r_1} \sigma_1^{p_2} \sigma_2^{-q_2} \sigma_3^{r_2}$ for $p_i,q_i,r_i >0$. Then its Alexander polynomial (up to some unit in $\mathbb{Z}[s^{1/2},s^{-1/2}])$ is given by:
	$$\Delta_L(s) = s[q_1][q_2]\left([p_1][p_2][r_1+r_2]+[r_1][r_2][p_1+p_2]\right) + [p_1+p_2][q_1+q_2][r_1+r_2].$$
	Moreover $\Delta_L(s)$ is trapezoidal, and the length of its stable part is  $l \leq |\sigma(L)|$.
\end{thm}

Recall that  the determinant of a  link $L$ is a numerical invariant of links that is   obtained by the evaluation of the Alexander polynomial at $t=-1$; $\det(L)=|\Delta_L(-1)|$. As a consequence of the above theorem, we have:\\
\begin{cor}
	Let $L$ be the alternating link obtained through closure of the 4-braid $\sigma_1^{p_1} \sigma_2^{-q_1} \sigma_3^{r_1} \sigma_1^{p_2} \sigma_2^{-q_2} \sigma_3^{r_2} $ for $p_i,q_i,r_i >0$. Then
	$$\det(L)= q_1q_2\left(p_1p_2(r_1+r_2)+r_1r_2(p_1+p_2)\right) + (p_1+p_2)(q_1+q_2)(r_1+r_2).$$

\end{cor}

\begin{proof}
	Let $\beta = \sigma_1^{p_1} \sigma_2^{-q_1} \sigma_3^{r_1} \sigma_1^{p_2} \sigma_2^{-q_2} \sigma_3^{r_2}.$
	
	For any positive integers $p_i, q_i$ and $r_i$, one can easily see that we have:
	
	$\psi_{4,s}(\sigma_1 ^{p_i})$ = $\begin{pmatrix}
		s^{p_i} & [p_i]  & 0 \\
		0 & 1 & 0 \\
		0 & 0 & 1
	\end{pmatrix}$,

	$\psi_{4,s}(\sigma_2 ^{-q_i})$ = $\begin{pmatrix}
		1 & 0 & 0 \\
		s^{1-q_i}[q_i] & s^{-q_i} & -s^{-q_i}[q_i] \\
		0 & 0 & 1
	\end{pmatrix}$, and

	$\psi_{4,s}(\sigma_3 ^{r_i})$ = $\begin{pmatrix}
		1 & 0 & 0 \\
		0 & 1 & 0 \\
		0 & -s[r_i] & s^{r_i}
	\end{pmatrix}.$
	
	Moreover, elementary matrix computations  show that we have:
	\begin{align*}
		det(I_3 - \psi_{4,s}(\beta)) &= \frac{s^{-q_1-q_2}}{(s-1)^5} \left(1-s+s^2-s^3\right) (s^{p_1+1}+s^{p_2+1}+s^{p_1+p_2}-2 s^{p_1+p_2+1}+s^{p_1+p_2+2}\\
		&+2 s^{q_1+1} -s^{p_1+q_1+1}-s^{p_2+q_1+1}+2 s^{q_2+1}-s^{p_1+q_2+1}-s^{p_2+q_2+1}+s^{q_1+q_2}\\
		&-4s^{q_1+q_2+1}+s^{q_1+q_2+2}+s^{p_1+q_1+q_2+1}+s^{p_2+q_1+q_2+1}-s^{p_1+p_2+q_1+q_2}\\
		&+2s^{p_1+p_2+q_1+q_2+1}-s^{p_1+p_2+q_1+q_2+2}+s^{r_1+1}-s^{p_1+p_2+r_1+1}-s^{q_1+r_1+1}\\
		&+s^{p_1+p_2+q_1+r_1+1}-s^{q_2+r_1+1}+s^{p_1+p_2+q_2+r_1+1}+s^{q_1+q_2+r_1+1}\\
		&-s^{p_1+p_2+q_1+q_2+r_1+1}+s^{r_2+1}-s^{p_1+p_2+r_2+1}-s^{q_1+r_2+1}+s^{p_1+p_2+q_1+r_2+1}\\
		&-s^{q_2+r_2+1}+s^{p_1+p_2+q_2+r_2+1}+s^{q_1+q_2+r_2+1}-s^{p_1+p_2+q_1+q_2+r_2+1}+s^{r_1+r_2}\\
		&-2 s^{r_1+r_2+1}+s^{r_1+r_2+2}-s^{p_1+r_1+r_2+1}-s^{p_2+r_1+r_2+1}\\
		&-s^{p_1+p_2+r_1+r_2}+4s^{p_1+p_2+r_1+r_2+1}-s^{p_1+p_2+r_1+r_2+2}+s^{p_1+q_1+r_1+r_2+1}\\
		&+s^{p_2+q_1+r_1+r_2+1}-2s^{p_1+p_2+q_1+r_1+r_2+1}+s^{p_1+q_2+r_1+r_2+1}+s^{p_2+q_2+r_1+r_2+1}\\
		&-2s^{p_1+p_2+q_2+r_1+r_2+1}-s^{q_1+q_2+r_1+r_2}+2s^{q_1+q_2+r_1+r_2+1}-s^{q_1+q_2+r_1+r_2+2}\\
		&-s^{p_1+q_1+q_2+r_1+r_2+1}-s^{p_2+q_1+q_2+r_1+r_2+1}+s^{p_1+p_2+q_1+q_2+r_1+r_2}\\
		&+s^{p_1+p_2+q_1+q_2+r_1+r_2+2}-s^2-1 ).
	\end{align*}
	
	By collecting the terms and simplifying we  get:
	\begin{align*}
		det(I_3 - \psi_{4,s}(\beta)) 	&= s^{-q_1-q_2} \left(1-s+s^2-s^3\right) (s[q_1][q_2] ([p_1][p_2][r_1+r_2] + [r_1][r_2][p_1+p_2]) \\
		&+ [p_1+p_2][r_1+r_2][q_1 + q_2]).
	\end{align*}
	
	Allowing us to obtain the  Alexander polynomial of $L$ up to some unit in $\mathbb{Z}[s^{1/2},s^{-1/2}]$. Indeed,  dividing
$det(I_3 - \psi_{4,s}(\beta))$ by $\frac{1-s^4}{1+s} = 1-s+s^2-s^3$, we get:
	$$\Delta_L(s) = s^{-q_1-q_2} \left(s[q_1][q_2] \left(\ [p_1][p_2][r_1+r_2] + [r_1][r_2][p_1+p_2]\ \right)+ [p_1+p_2][r_1+r_2][q_1 + q_2]\right).$$
	Notice that $s[q_1][q_2] \left(\ [p_1][p_2][r_1+r_2] + [r_1][r_2][p_1+p_2]\ \right)$ is a trapezoidal symmetric polynomial with center $\frac{q_1+q_2+p_1+p_2+r_1+r_2-3}{2}$. Likewise $[p_1+p_2][r_1+r_2][q_1 + q_2]$ is a trapezoidal symmetric polynomial with the same center. Therefore by Lemma  \ref{addsymm}, $\Delta_L(s)$ is a trapezoidal symmetric polynomial, see the illustration in Figure \ref{Trapezoidal1} and Figure \ref{Trapezoidal2}.
	
To complete the proof, it will  be enough to show that $[p_1+p_2][r_1+r_2][q_1 + q_2]$ has  length of  stable part  less than or equal $|\sigma(L)|$. By letting $P = p_1 +p_2, R=r_1+r_2$ and  $Q = q_1+q_2$, we have  $\sigma(L) = Q-P-R+1$. This is exactly the case we  proved in Theorem \ref{case1}. By   Lemma  \ref{addsymm},  the length of the stable part
of $\Delta_L(s)$ is the  minimum between the lengths of the stable parts of the two terms. Since the term  $[p_1+p_2][r_1+r_2][q_1 + q_2]$ has the length of  stable part less than or equal to $ |\sigma(L)|$, we conclude that    $l \leq |\sigma(L)|$.
\end{proof}

\begin{figure}[h]
	\begin{center}
		\includegraphics[width =\linewidth]{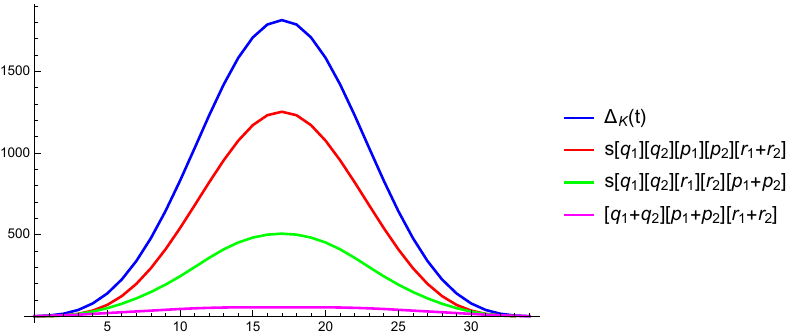}
	\end{center}
	\caption{Coefficients of $\Delta_K(s)$ plotted for a knot $K = cl(\sigma_1^{p_1} \sigma_2^{-q_1} \sigma_3^{r_1} \sigma_1^{p_2} \sigma_2^{-q_2} \sigma_3^{r_2})$.}
\label{Trapezoidal2}
\end{figure}
\section{The first coefficients of the Alexander polynomial}
In this section, we shall use the  Burau representation to give explicit  formulas for the first 4 coefficients of the Alexander polynomial
of some classes of  alternating  $n$-braid links.
\begin{thm}
	Let $L$ be the link obtained through the closure of the 4-braid on $m$ blocks $\sigma_{1}^{p_1} \sigma_2^{-q_1} \sigma_{3}^{r_1}... \sigma_{1}^{p_m} \sigma_2^{-q_m} \sigma_{3}^{r_m}$. Then  $\Delta_L(s)=\pm \displaystyle\sum_{i=0}^{2n}a_is^i$, where:
	\begin{enumerate}
		\item if $p_i, q_i, r_i > 1$, then  $a_0 = 1$ and  $a_1 = 2m+1$,
		\item if $p_i, q_i, r_i > 2$, then $a_2 = 2m^2 +4m$,
		\item if $p_i, q_i, r_i > 3$, then $a_3 = \frac{2}{3} m \left(2 m^2+9 m+4\right)$.
	\end{enumerate}
\end{thm}

\begin{proof}
	Let $\alpha_m$ denote  the 4-braid on $m$ blocks
 $\alpha_m = \sigma_{1}^{p_1} \sigma_2^{-q_1} \sigma_{3}^{r_1}... \sigma_{1}^{p_m} \sigma_2^{-q_m} \sigma_{3}^{r_m}$.
	
	Elementary computations show  that for all $1\leq i \leq m$ we have:
	$$\Psi_{4,s}(\sigma_{1}^{p_i} \sigma_2^{-q_i} \sigma_{3}^{r_i}) = s^{-q_i} \begin{pmatrix}
		s^{p_i+q_i}+[p_i] [q_i] s & [p_i] [q_i] [r_i] s+[p_i]  & [p_i] [q_i] \left(-s^{r_i}\right) \\
		[q_i] s & [q_i] [r_i] s+1 & [q_i] \left(-s^{r_i}\right) \\
		0 & [r_i] \left(-s^{q_i+1}\right) & s^{q_i+r_i}
	\end{pmatrix}.$$

	It easy to see  that the first four coefficients of $s^Q \det(I_3 - \psi_{4,s}(\alpha_m)) (1+s)$, where $Q = \sum_{i=1}^{n} q_i$, are equal to the first four coefficients of the Alexander polynomial of ${\rm cl}(\alpha_m)$. Let us first compute   $\det(s^Q I_3 - s^Q\psi_{4,s}(\alpha_m))$. It is clear   that if $p_i, q_i, r_i > 1$, then taking the matrix  entries modulo $s^2$ we get:
	$$\Psi_{4,s}(\sigma_{1}^{p_i} \sigma_2^{-q_i} \sigma_{3}^{r_i}) \equiv s^{-q_i} \begin{pmatrix}
		s & 1 + 2 s  & 0 \\
		s & 1 + s & 0 \\
		0 & 0 & 0
	\end{pmatrix}.$$
	
	Thus, modulo $s^2$, we have:
	
	\begin{align*}
		\det( s^QI_3 - s^Q \Psi_{4,s}(\alpha_m)) &\equiv \det\left(\begin{pmatrix}
			s^Q & 0  & 0 \\
			0 & s^Q & 0 \\
			0 & 0 & s^Q
		\end{pmatrix} - \begin{pmatrix}
		s & 1 + 2 s  & 0 \\
		s & 1 + s & 0 \\
		0 & 0 & 0
	\end{pmatrix}^m\right) \\
		&\equiv 4^{-m} s^Q \left(2^m s^Q-\left(-\sqrt{1+4s}+2 s+1\right)^m\right) \left(2^m s^Q-\left(\sqrt{1+4s}+2 s+1\right)^m\right)\\
		&\equiv -2^{-m}s^{2 Q}\left( \left(2 s+1-\sqrt{1+4s}\right)^m + \left(2 s+1+\sqrt{1+4s}\right)^m \right).
	\end{align*}
	
	Computing $\left(2 s+1-\sqrt{1+4s}\right)^m + \left(2 s+1+\sqrt{1+4s}\right)^m$ modulo $s^2$ we get that:
	$$\left(2 s+1-\sqrt{1+4s}\right)^m + \left(2 s+1+\sqrt{1+4s}\right)^m \equiv 2^m (1+2ms).$$
	
	Thus, since $s^Q \det(I_3 - \psi_{4,s}(\alpha_m)) = s^{-2Q}\det( s^QI_3 - s^Q \Psi_{4,s}(\alpha_m))$, we conclude  that  the first two coefficients of the Alexander polynomial of the closure of $\alpha_m$ are   $a_0 = 1$ and $ a_1 = 2m+1$.
	
	Similarly, if   $p_i, q_i, r_i > 2$, then taking the matrix modulo $s^3$ we get that:
	$$\det( s^QI_3 - s^Q \Psi_{4,s}(\alpha_m)) \equiv s^{2Q} (1+2ms + (2m^2+2m)s^2 ),$$
which gives $a_2=2m^2+4m$. Furthermore, taking $p_i, q_i, r_i > 3$, then considering the matrix  modulo $s^4$ we get that:
	$$\det( s^QI_3 - s^Q \Psi_{4,s}(\alpha_m)) \equiv s^{2Q} (1+2ms + (2m^2+2m)s^2 + \frac{2}{3} (2 m^3+6 m^2+m)s^3).$$ Hence,
 $a_3=\frac{2}{3} m \left(2 m^2+9 m+4\right)$.
\end{proof}

\begin{rem}
\begin{enumerate}
\item The sequence made up of the  first 4 coefficients above  is clearly logarithmic concave. Indeed, one can easily verify that
$a_1^2\geq |a_0||a_2|$ and that $a_2^2\geq |a_1||a_3|$.\\
\item Similar computations carried on  alternating braids with $n$ strands and $m$ blocks  show that the first coefficients of the Alexander polynomial of the closure of the braid
$$\beta_{n,m}((p_{i,j}))=\Pi_{j=1}^m \left( \sigma_1^{p_{1,j}} \sigma_2^{-p_{2,j}} \sigma_3^{p_{3,j}} ... \sigma_{n-1}^{(-1)^{n-2} p_{{n-1},j}} \right),$$
are given, for large values of $p_{i,j}$,  by the following formulas:\\
$a_0=1$, $a_1=(n-2)m+1$, $a_2=\displaystyle\frac{(n-2)^2m^2+(3(n-2)+2)m}{2}$ and for $ n>3$ we have $a_3=\displaystyle\frac{(n-2)^3m^3+6(n-1)(n-2)m^2+(5(n-1)+1))m}{6}$.
For $n=3$, using the computations on the Jones polynomials in \cite{Ch2022}, we get $a_3=\displaystyle\frac{m^3+12m^2+17m-6}{6}$.\\
Once again, it can be verified that the sequence made up of these 4 coefficients is log-concave. Indeed, for $n\geq 4$, we have:

$$\begin{array}{rl}
a_1^2- |a_0||a_2|=&\displaystyle\frac{2 + m ( n-4) + m^2 ( n-2)^2 }{2}\geq 0, \mbox{ and } \\
a_2^2 -|a_1||a_3|=&
\displaystyle\frac{m (4 m^2 ( n-2)^3 + m^3 ( n-2)^4 - 10 n +
   m (5 n^2-16n+24))}{12}\geq 0.
   \end{array}$$
\end{enumerate}
\end{rem}

\section*{Acknowledgments}
 This research was funded  by  United Arab Emirates University, UPAR grant
 $\# G00004167.$

\end{document}